\documentclass[12pt]{article}
\usepackage{amssymb} 
\usepackage{amsmath}
\usepackage{amsthm}
\usepackage{enumerate}
\numberwithin{equation}{section}
\newtheorem{theorem}{Theorem}[section]
\newtheorem{proposition}[theorem]{Proposition}
\newtheorem{corollary}[theorem]{Corollary}
\newtheorem{lemma}[theorem]{Lemma}
\newtheorem{remark}[theorem]{Remark}
\newtheorem{example}[theorem]{Example}
\newtheorem{case}{Case}
\newtheorem{subcase}{Case}
\numberwithin{subcase}{case}
\newtheorem{subsubcase}{Case}
\numberwithin{subsubcase}{subcase}
\usepackage{longtable}
\usepackage{authblk}
\usepackage{caption}

\date{}

\title{An almost effective proof of the class number 1 problem for imaginary quadratic fields via Frobenius--Rabinowitsch}

\author{James E. Carter\thanks{carterj@cofc.edu}}
\affil{Dept. of Mathematics, College of Charleston, Charleston, SC 29401}

\begin{document}

\maketitle

\begin{abstract}

Let $D \neq 1$ be a square-free integer, and let $K = {\mathbb Q}(\sqrt D)$ be the corresponding quadratic field with ring of integers $\mathcal{O}_K$. Let $\mathcal{P}$ be a prime ideal of $\mathcal{O}_K$ lying over an odd rational prime that splits in $K/\mathbb Q$. We obtain a necessary and sufficient condition in terms of a binary quadratic form for $\mathcal{P}$ to be a principal ideal. Then, when $K$ is imaginary with $D\equiv 1\pmod{4}$ and $|D|>16$, we use this result to prove $K$ has class number 1 only if $D=1-4p$ where $p$, $4p-1$, and $4p+3$ are prime. Finally, combining this condition with results on the representation of integers by the sum of three squares, the theorem of Frobenius and Rabinowitsch, and a theorem of Tatuzawa we provide a proof, up to at most one possible additional field, of the well-known theorem on imaginary quadratic fields with class number 1. If Dirichlet $L$-functions have no zeros in the half-plane $\mathrm{Re}(s)>11/12$, then the possible additional field of our result does not exist.

\end{abstract}

\section{Introduction}

\noindent Suppose $K$ is a number field with ring of integers ${\mathcal O}_K$. 
If $I$ is an ideal of ${\mathcal O}_K$ then $I$ can be generated by at most two elements of ${\mathcal O}_K$, that is, $I=\alpha{\mathcal O}_K$ for some element $\alpha \in {\mathcal O}_K$, or $I=\alpha{\mathcal O}_K+\beta{\mathcal O}_K$ for some elements $\alpha, \beta \in {\mathcal O}_K$ (see, for instance, Theorem 17, p. 61 of [10]). In the former case we write $I=(\alpha)$ and call  $I$ a principal ideal. In the latter we write $I=(\alpha, \beta)$. Now let $D \neq 1$ be a square-free integer and let $K={\mathbb Q}(\sqrt D)$. In this paper we will obtain a necessary and sufficient condition for a prime ideal of   ${\mathcal O}_K$ lying over an odd rational prime that splits in $K/\mathbb Q$ to be principal. Recall that if $q$ is an odd prime and $q$ does not divide $D$, then 

\[q{\mathcal O}_K=\begin{cases}
(q, n+\sqrt D)(q, n-\sqrt D)\   \text{if } D\equiv n^2 \pmod q \\ 
(q) \ \text{ if } D \text{ is not a square mod } q
\end{cases}\]

\noindent where $(q, n+\sqrt D)$, $(q, n-\sqrt D)$, and $(q)$ are prime ideals of ${\mathcal O}_K$ (see, for instance, Theorem 25, p. 74 of [10]). We will prove the following theorem.

\begin{theorem}
Let $D\neq 1$ be a square-free integer and let $K={\mathbb Q}(\sqrt D)$. Let $\delta \in \{1,2\}$ with $\delta=2$ if and only if $D\equiv 1 \pmod 4$. Suppose $q$ is an odd prime such that $q$ does not divide $D$ and $D\equiv n^2 \pmod q$. Then $n^2-D=lq$ for some $l \in \mathbb Z$ and we have

\begin{enumerate}[(i)]

 \item  if $D<0$ then the prime ideal $\mathcal P=(q, n+\sqrt D)$ is a principal ideal if and only if the binary quadratic form $f_{\mathcal P}(x,y)=lx^2+2nxy+qy^2$ represents $\delta^2$.

\item if $D>0$ then the prime ideal $\mathcal P=(q, n+\sqrt D)$ is a principal ideal if and only if the binary quadratic form $f_{\mathcal P}(x,y)=lx^2+2nxy+qy^2$ represents $\pm \delta^2$.

\end{enumerate}
\end{theorem}

\begin{remark}
\normalfont Following the notation and terminology established in art. 153 and art. 154 of [5], we call $D=n^2-lq$ the determinant of the binary quadratic form $f_{\mathcal P}(x,y)=lx^2+2nxy+qy^2$. In the remainder of the paper we will omit the subscript ${\mathcal P}$ on $f$ since no confusion can arise by doing so.
\end{remark}
Examples illustrating Theorem 1.1 are presented in Section 5 below. The methods illustrated in these examples, together with results on the representation of integers by the sum of three squares, the theorem of Frobenius and Rabinowitsch, and a theorem of Tatuzawa are then used in Section 6  to give a proof, up to at most one possible additional value of $D$, of the following well-known theorem, which was proved in [1], [9], and [16]. Aspects of the interesting history of these proofs can be found, for instance, in [2], [7], [16], and [17]. 
\begin{theorem}[Baker--Heegner--Stark]
The imaginary quadratic fields with class number 1 are exactly the fields ${\mathbb Q}(\sqrt D)$ where $D$ equals  $-1$, $-2$, $-3$, $-7$, $-11$, $-19$, $-43$, $-67$, or $-163$.
\end{theorem}
The proof presented here reveals some new information regarding necessary conditions for an imaginary quadratic field to have class number 1 when $D\equiv 1 \pmod 4$ and $|D|>16$ (see Proposition 6.5 below).  If the recent result claiming the half-plane $\mathrm{Re}(s)>11/12$ contains no zeros of any Dirichlet $L$-function is true (see, for instance, [12]), then the possible additional value of $D$ of our result does not exist (see Corollary 6.14 and Remark 6.15 below).

\section{Elements of Ideals and Their Norms}

\noindent Let $(q, n+\sqrt D)$  be as described in the statement of Theorem 1.1. Denote by $N_{K/\mathbb Q}$ the norm of elements from $K$ to $\mathbb Q$. The next lemma provides a description of the elements of $(q, n+\sqrt D)$ and their norms.

\begin{lemma}
Let $\gamma \in (q, n+\sqrt D)$. Then 

\begin{longtable}{ll} 
 $ \displaystyle \gamma = \frac{qa+nc+dD+(qb+c+nd)\sqrt D}{\delta}$,
\end{longtable}
\noindent and 

\begin{longtable}{ll} 
$\displaystyle N_{K/\mathbb Q}(\gamma)=\frac{q^2a^2-q^2Db^2+2q(nc+dD)a-2qD(c+nd)b+(c^2-d^2D)lq}{\delta^2}$ 
\end{longtable}
\noindent where $a,b,c,d \in \mathbb Z$, $\delta \in \{1,2\}$, and $n^2-D=lq$ for some $l\in {\mathbb Z}$. Furthermore, if $D\equiv 2 \ \text {or}\ 3 \pmod 4$,  then $\delta = 1$; if $D\equiv 1 \pmod 4$, then $\delta = 2$, $a\equiv b \pmod 2$, and $ c\equiv d \pmod 2$. 
 
\end{lemma}
 
\begin{proof}

If  $D\equiv 2 \ \text {or}\ 3 \pmod 4$, then
 
\begin{longtable}{ll} 
${\mathcal O}_K=\{r+s\sqrt D: r,s\in \mathbb Z\}$,
\end{longtable}

\noindent and if $D\equiv 1 \pmod 4$, then 

\begin{longtable}{ll} 
$\displaystyle {\mathcal O}_K=\left\{\frac{r+s\sqrt D}{2}: r,s\in \mathbb Z, \ \text{with}\ r\equiv s\pmod 2\right\}$
\end{longtable}

\noindent (see, for instance, Theorem 3.6, p. 22 of [13]). Hence,

\begin{longtable}{ll} 
$\gamma$&$\displaystyle= q\left( \frac{a+b\sqrt D}{\delta}\right )+(n+\sqrt D)\left (\frac{c+d\sqrt D}{\delta}\right )$ \\ \\
 &$\displaystyle= \frac{qa+nc+dD+(qb+c+nd)\sqrt D}{\delta},$
 \end{longtable}
 
\noindent where $a,b,c,d \in \mathbb Z$, and $\delta \in \{1,2\}$. Furthermore, if $D\equiv 2 \ \text {or}\ 3 \pmod 4$,  then $\delta = 1$; if $D\equiv 1 \pmod 4$, then $\delta = 2$, $a\equiv b \pmod 2$, and $ c\equiv d \pmod 2$. Computing the norm of $\gamma$ we obtain 
 
\begin{longtable}{lll}
$N_{K/\mathbb Q}(\gamma)$&=&$[(qa+nc+dD)^2-(qb+c+nd)^2D]/\delta^2$ \\ \\
&=&$\{[q^2a^2+2qa(nc+dD)+(nc+dD)^2]$ \\ \\
&$-$&$[q^2b^2+2qb(c+nd)+(c+nd)^2]D\}/ \delta^2$ \\ \\
&=&$\{[q^2a^2+2qanc+2qadD+n^2c^2+2ncdD+d^2D^2]$ \\ \\ 
&$-$&$[q^2b^2+2qbc+2qbnd+c^2+2cnd+n^2d^2]D\}/ \delta^2$ \\ \\
&=&$\{[q^2a^2+2qanc+2qadD+n^2c^2+2ncdD+d^2D^2]$ \\ \\ 
&$-$&$[q^2b^2D+2qbcD+2qbndD+c^2D+2cndD+n^2d^2D]\}/ \delta^2$ \\ \\
&=&$\{q^2a^2-q^2Db^2+(2qnc+2qdD)a-(2qcD+2qndD)b$ \\ \\ 
&+&$(n^2c^2-c^2D)+(d^2D^2-n^2d^2D)\}/\delta^2$ \\ \\
&=&$[q^2a^2-q^2Db^2+2q(nc+dD)a-2qD(c+nd)b$ \\ \\
&+&$c^2(n^2-D)+d^2D(D-n^2)]/\delta^2$ \\ \\
&=&$[q^2a^2-q^2Db^2+2q(nc+dD)a-2qD(c+nd)b$ \\ \\
&+&$(c^2-d^2D)(n^2-D)]/\delta^2$ \\ \\
&=&$[q^2a^2-q^2Db^2+2q(nc+dD)a-2qD(c+nd)b$ \\ \\ 
&+&$(c^2-d^2D)lq]/\delta^2.$
\end{longtable}
 
 \end{proof}
 
  \section{A Quadratic Diophantine Equation}

\begin{lemma}

Let the notation be as in the statement of Lemma 2.1. Then $(q, n+\sqrt D)=(\gamma)$ if and only if the quadratic Diophantine equation in the unknowns $x$ and $y$

\begin{equation}
qx^2-qDy^2+2(nc+dD)x-2D(c+nd)y+(c^2-d^2D)l\pm\delta^2=0
\end{equation} 

\noindent has a solution in integers $x=a$ and $y=b$ for some integers $c$ and $d$. Furthermore, if $D\equiv 2 \ \text {or}\ 3 \pmod 4$,  then $\delta = 1$; if $D\equiv 1 \pmod 4$, then $\delta = 2$, $a\equiv b \pmod 2$, and $ c\equiv d \pmod 2$. 

\end{lemma}

\begin{proof} 

Suppose $(q, n+\sqrt D)=(\gamma)$. Taking the ideal norm (see, for instance, Section 4.1, p. 140 of [11]) of both sides of this equation and using Corollary 1, p. 142 of [11] we have $(q)=(N_{K/\mathbb Q}(\gamma))$ which implies $N_{K/\mathbb Q}(\gamma)=\pm q$. Conversely, suppose $N_{K/\mathbb Q}(\gamma)=\pm q$. Then $(N_{K/\mathbb Q}(\gamma))=(q)$. Since $(\gamma)\subseteq (q, n+\sqrt D)$ we have $(\gamma)=I(q, n+\sqrt D)$ for some ideal $I$ of ${\mathcal O}_K$ (see, for instance,  Corollary 3, p. 59 of [10]). Taking the ideal norm of both sides of this equation we have $(N_{K/\mathbb Q}(\gamma))=(\alpha)(q)$ for some $\alpha \in \mathbb Z$. Hence, $\alpha =\pm1$ which implies that $I={\mathcal O}_K$, so $(q, n+\sqrt D)=(\gamma)$. We have shown

\begin{equation}
(q, n+\sqrt D)=(\gamma) \Longleftrightarrow (q)=(N_{K/\mathbb Q}(\gamma))\Longleftrightarrow N_{K/\mathbb Q}(\gamma)=\pm q.
\end{equation}

\noindent By Lemma 2.1, the last equation is equivalent to

 \begin{longtable}{l}
$q^2a^2-q^2Db^2+2q(nc+dD)a-2qD(c+nd)b+(c^2-d^2D)lq\pm\delta^2q=0$
\end{longtable}

\noindent which is equivalent to

 \begin{longtable}{l}
$qa^2-qDb^2+2(nc+dD)a-2D(c+nd)b+(c^2-d^2D)l\pm\delta^2=0.$ 
\end{longtable}

\noindent Furthermore, if $D\equiv 2 \ \text {or}\ 3 \pmod 4$,  then $\delta = 1$; if $D\equiv 1 \pmod 4$, then $\delta = 2$, $a\equiv b \pmod 2$, and $ c\equiv d \pmod 2$. 

\end{proof}

\begin{remark}

\normalfont (i) We note that when $D\equiv 1 \pmod 4$, so $\delta=2$, we have $c\equiv d \pmod 2$. Hence $c^2-d^2\equiv 0 \pmod 2$. Therefore any solution in integers $a$ and $b$ to (3.1) in this case satisfies $a \equiv b \pmod 2$ as is seen by setting $x=a$ and $y=b$ in (3.1) and reducing both sides of this equation modulo 2. (ii) If $D<0$ then all nonzero elements of $K$ have positive norm so we only get the minus sign on $\delta^2$ in (3.1).

\end{remark}

Following art. 216 of [5], to find all integer solutions to (3.1), if any, we first replace the unknowns $x$ and $y$ by new ones $w$ and $z$ defined by

\begin{longtable}{l} 
$ w=q^2Dx+qD(nc+dD)\ \ {\rm and}\ \ z=q^2Dy+qD(c+nd).$
\end{longtable}

\noindent Then we rewrite (3.1) in terms of these new unknowns to obtain

\begin{longtable}{l} 
$qw^2-qDz^2-M=0$
\end{longtable}
\noindent where 

\begin{longtable}{lll} 
$-M$&=&$[(c^2-d^2D)l\pm\delta^2](q^2D)^2+q^2D\{q[-D(c+nd)]^2$ \\ \\
&+&$(-qD)(nc+dD)^2\}$ \\ \\
&=&$\{[(c^2-d^2D)l\pm\delta^2]q^2D+q[-D(c+nd)]^2+(-qD)(nc+dD)^2\}q^2D$ \\ \\ 
&=&$\{[(c^2-d^2D)l\pm\delta^2]q^2D+qD^2(c+nd)^2+(-qD)(nc+dD)^2\}q^2D$ \\ \\ 
&=&$\{[(c^2-d^2D)l\pm\delta^2]q+D(c+nd)^2-(nc+dD)^2\}q^3D^2$ \\ \\ 
&=&$\{[(c^2-d^2D)l\pm\delta^2]q+D(c^2+2cnd+n^2d^2)-(n^2c^2+2ncdD$ \\ \\
&+&$d^2D^2)\}q^3D^2$ \\ \\ 
&=&$\{[(c^2-d^2D)l\pm\delta^2]q+Dc^2+D2cnd+Dn^2d^2-n^2c^2-2ncdD$ \\ \\
&$-$&$d^2D^2\}q^3D^2$ \\ \\ 
&=&$\{[(c^2-d^2D)l\pm\delta^2]q+Dc^2-n^2c^2+Dn^2d^2-d^2D^2\}q^3D^2$ \\ \\
&=&$\{[(c^2-d^2D)l\pm\delta^2]q+c^2(D-n^2)+Dd^2(n^2-D)\}q^3D^2$ \\ \\
&=&$\{[(c^2-d^2D)l\pm\delta^2]q+c^2(-lq)+Dd^2(lq)\}q^3D^2$  \\ \\
&=&$\{lc^2-ld^2D\pm\delta^2-lc^2+lDd^2\}q^4D^2$  \\ \\
&=&$\pm\delta^2q^4D^2.$    
\end{longtable}

\noindent Now, all integer solutions to (3.1), if any, are given by

\begin{equation}
x=\frac{r-qD(nc+dD)}{q^2D}\ \ {\rm and}\ \  y=\frac{s-qD(c+nd)}{q^2D}
\end{equation}
where $w=r$ and $z=s$ are integers which give a solution to 

\begin{equation}
qw^2-qDz^2=\pm\delta^2q^4D^2.
\end{equation}

\section{Proof of Theorem 1.1}

\begin{proof} 

Let the notation be as in the statement of Theorem 1.1. We first prove the necessity of the condition.
Suppose $(q, n+\sqrt D)$ is a principal ideal. Then by Lemma 3.1 there are integers $a$ and $b$ such that  $x=a$ and $y=b$ is a solution to (3.1) for some integers $c$ and $d$. Furthermore, if $D\equiv 2 \ \text {or}\ 3 \pmod 4$,  then $\delta = 1$; if $D\equiv 1 \pmod 4$, then $\delta = 2$, $a\equiv b \pmod 2$, and $ c\equiv d \pmod 2$. Moreover, by (3.3) we have
 
\begin{longtable}{l} 
 $\displaystyle a=\frac{r-qD(nc+dD)}{q^2D}\ \ {\rm and}\ \  b=\frac{s-qD(c+nd)}{q^2D}$
\end{longtable}

\noindent where $w=r$ and $z=s$ are integers that give a solution to (3.4). Solving these two equations for $r$ and $s$, respectively, we see that $r=mqD$ and $s=kqD$ for some integers $m$ and $k$. Therefore,

\begin{longtable}{l} 
$\displaystyle a=\frac{mqD-qD(nc+dD)}{q^2D}=\frac{m-(nc+dD)}{q},$
\end{longtable}

\noindent so 

\begin{longtable}{ll} 
&$nc+dD\equiv m \pmod q$ \\ \\
$\Longleftrightarrow$&$nc+d(n^2-lq)\equiv m \pmod q$  \\ \\
$\Longleftrightarrow$&$nc+dn^2\equiv m \pmod q.$ 
\end{longtable}

\noindent Hence,

\begin{equation}
n(c+nd)\equiv m \pmod q. 
\end{equation}

\noindent Also,

\begin{longtable}{l} 
$\displaystyle b=\frac{kqD-qD(c+nd)}{q^2D}=\frac{k-(c+nd)}{q},$
\end{longtable}

\noindent so $c+nd\equiv k \pmod q$. Therefore,

\begin{equation}
n(c+nd)\equiv kn \pmod q.
\end{equation}

\noindent From (4.1) and (4.2) we deduce that $m\equiv kn \pmod q$. So $r=(kn+vq)qD$ for some $v\in \mathbb Z$. Since setting $w=r$ and $z=s$ gives a solution to (3.4) we have

\begin{longtable}{ll}&$q((kn+vq)qD)^2-qD(kqD)^2=\pm\delta^2 D^2q^4$  \\ \\
$\Longleftrightarrow$ & $q(kn+vq)^2(qD)^2-qDk^2(qD)^2=\pm \delta^2 D^2q^4$ \\ \\
$\Longleftrightarrow$ & $(q(kn+vq)^2-qDk^2)(qD)^2=\pm\delta^2 D^2q^4$  \\ \\
$\Longleftrightarrow$ & $((kn+vq)^2-Dk^2)q(qD)^2=\pm\delta^2 D^2q^4$  \\ \\
$\Longleftrightarrow$ & $(k^2n^2+2knvq+v^2q^2-Dk^2)q(qD)^2=\pm\delta^2 D^2q^4$ \\ \\
$\Longleftrightarrow$ & $(k^2n^2-Dk^2+2knvq+v^2q^2)q(qD)^2=\pm\delta^2 D^2q^4$ \\ \\ 
$\Longleftrightarrow$ & $(k^2(n^2-D)+2knvq+v^2q^2)q(qD)^2=\pm\delta^2 D^2q^4$ \\ \\
$\Longleftrightarrow$ & $(k^2(lq)+2knvq+v^2q^2)q(qD)^2=\pm\delta^2 D^2q^4$ \\ \\
$\Longleftrightarrow$ & $(k^2l+2knv+v^2q)q^2(qD)^2=\pm\delta^2 D^2q^4$ \\ \\  
$\Longleftrightarrow$ & $(k^2l+2knv+v^2q)D^2q^4=\pm\delta^2 D^2q^4$ \\ \\
$\Longleftrightarrow$ & $lk^2+2nkv+qv^2=\pm\delta^2.$
\end{longtable}

\noindent Hence, the binary quadratic form $lx^2+2nxy+qy^2$ represents $\pm\delta^2$. We now show that the condition is sufficient. Suppose the binary quadratic form $lx^2+2nxy+qy^2$ represents $\pm\delta^2$. Then there are integers $k$ and $v$ such that $lk^2+2nkv+qv^2=\pm\delta^2$. Hence, by the last series of equivalences, setting $w=(kn+vq)qD$ and $z=kqD$ gives a solution to (3.4). Using these values for $r$ and $s$, respectively, in (3.3) and setting $x=a$ and $y=b$ there we have 

\begin{longtable}{ll}
$a$&$\displaystyle=\frac{(kn+vq)qD-qD(nc+dD)}{q^2D}
\displaystyle=\frac{qD(kn+vq-nc-dD)}{q^2D}$  \\ \\
&$\displaystyle=\frac{kn+vq-nc-dD}{q} 
\displaystyle=\frac{kn+vq-nc-d(n^2-lq)}{q}$ \\ \\
&$\displaystyle=\frac{kn+vq-nc-dn^2+dlq}{q} 
\displaystyle=\frac{kn-n(c+dn)+vq+dlq}{q}$ \\ \\
&$\displaystyle=\frac{kn-n(c+dn)}{q}+\frac{vq+dlq}{q} 
\displaystyle=n\left (\frac{k-(c+nd)}{q}\right )+\frac{vq+dlq}{q},$
\end{longtable}

\noindent and 

\begin{longtable}{ll}$b$&$\displaystyle=\frac{kqD-qD(c+nd)}{q^2D} 
\displaystyle=\frac{qD(k-(c+nd))}{q^2D} 
\displaystyle=\frac{k-(c+nd)}{q}.$
\end{longtable}

\noindent We now show that we can choose $c$ and $d$ so that $a$ and $b$ are integers. Note that this will be the case if $c$ and $d$ are chosen such that 

\begin{longtable}{l}
$k-(c+nd)\equiv 0\pmod q \Longleftrightarrow 
nd\equiv k-c \pmod q.$
\end{longtable}

\noindent Since $n^2-D=lq$ and $q$ does not divide $D$, we have $q$ does not divide $n$. Hence,  the congruence has a unique solution $d$ modulo $q$ for any choice of $c$. Moreover, if $D\equiv 1\pmod 4$ we can choose the solution $d$ so that $d\equiv c \pmod 2$ since by the Chinese Remainder Theorem the following system of congruences has a unique solution $d$ modulo $2q$ for any choice of $c$

\begin{longtable}{ll}$x\equiv n^{-1}(k-c) \pmod q$ \\ \\
$x\equiv c \pmod 2$.\end{longtable}

\noindent By Remark 3.2 (i) $c\equiv d \pmod 2$ guarantees $a\equiv b \pmod 2$. A choice of integers $c$ and $d$ according to the procedure just described then determines a solution $x=a$ and $y=b$ to (3.1). Hence, by Lemma 3.1, $(q, n+\sqrt D)=(\gamma)$ where $\gamma$ is given by
Lemma 2.1.

\end{proof}

\section{Some Examples}

\begin{example} 

\normalfont Consider $K={\mathbb Q}(\sqrt 5)$. Since $5\not \equiv 0\pmod{101}$ and $(45)^2-5=20\cdot101$, the prime 101 splits in $K$. We have $101{\mathcal O}_K=(101, 45+\sqrt 5)(101, 45-\sqrt 5)$. Since $K$ has class number 1, the ideal  $(101, 45+\sqrt 5)$ is principal. We will find a generator for it. Since $5\equiv 1\pmod 4$ we have $\delta=2$, so following Theorem 1.1 (ii) we consider the equation
$20x^2+2(45)xy+101y^2=\pm4$. When the right-hand side of this equation is 4, we find that  $x=-4$ and $y=2$ gives a solution to the equation. Following the proof of the sufficiency of the condition of Theorem 1.1, we take $k=-4$ and $v=2$ so 

\begin{longtable}{l}
$\displaystyle a=45\left (\frac{-4-(c+45\cdot d)}{101}\right )+\frac{2\cdot 101+d\cdot 20\cdot 101}{101}$
\end{longtable}

\noindent and

\begin{longtable}{l}
$\displaystyle b=\frac{-4-(c+45\cdot d)}{101}.$
\end{longtable}

\noindent Since $5\equiv 1\pmod 4$ we must have $c\equiv d \pmod 2$. Choosing $c=0$ we then need to solve the following system of congruences for $d$ 

\begin{longtable}{l}$d\equiv 45^{-1}(-4) \pmod {101}$ \\ \\
$d\equiv 0 \pmod 2,$
\end{longtable}

\noindent that is, we need to solve the system 

\begin{longtable}{l}$d\equiv 65 \pmod {101}$ \\ \\
$d\equiv 0 \pmod 2.$
\end{longtable}

\noindent Choosing the solution $d=166$ we have

\begin{longtable}{l}
$\displaystyle a=45\left (\frac{-4-(0+45\cdot 166)}{101}\right )+\frac{2\cdot 101+166\cdot 20\cdot 101}{101}=-8$
\end{longtable}

\noindent and

\begin{longtable}{l}
$\displaystyle b=\frac{-4-(0+45\cdot 166)}{101}=-74.$
\end{longtable}

\noindent Finally, from Lemma 2.1,
 
\begin{longtable}{l}
$\displaystyle \gamma = \frac{101(-8)+166\cdot 5+(101(-74)+45\cdot 166)\sqrt 5}{2}=11-2\sqrt{5}.$
\end{longtable}

\noindent Since $N_{K/\mathbb Q}(\gamma)=101$,  we have $(101, 45+\sqrt 5)=(\gamma)$ by (3.2).

\end{example}

\begin{example}

\normalfont Consider $K={\mathbb Q}(\sqrt {10})$. Since $10\not \equiv 0\pmod{71}$ and $9^2-10=1\cdot 71$, the prime 71 splits in $K$. We have $71{\mathcal O}_K=(71, 9+\sqrt {10})(71, 9-\sqrt {10})$. It is easy to see that $(71, 9+\sqrt{10})=( 9+\sqrt{10})$. We will now recover this fact using Theorem 1.1. Since $10\equiv 2\pmod 4$ we have $\delta=1$, so following Theorem 1.1 (ii) we consider the equation
$x^2+2(9)xy+71y^2=\pm 1$. When the right-hand side of this equation is 1, we find that $x=1$ and $y=0$ gives a solution to the equation. Hence, by Theorem 1.1 (ii), the ideal $(71, 9+\sqrt {10})$ is principal. We will now find a generator for it. Following the proof of the sufficiency of the condition of Theorem 1.1, we take $k=1$ and $v=0$ so 

\begin{longtable}{l}
$\displaystyle a=9\left (\frac{1-(c+9\cdot d)}{71}\right )+\frac{0\cdot 71+d\cdot 1\cdot 71}{71}$
\end{longtable}

\noindent and

\begin{longtable}{l}
$\displaystyle b=\frac{1-(c+9\cdot d)}{71}.$
\end{longtable}

\noindent Since $10\equiv 2\pmod 4$ we only solve the congruence 

\begin{longtable}{l}
$9d\equiv 1-c \pmod {71}$
\end{longtable}

\noindent for $d$ for any choice of $c$. Taking $c=1$ we obtain $d=0$. So

\begin{longtable}{l}
$\displaystyle a=9\left (\frac{1-(1+9\cdot 0)}{71}\right )+\frac{0\cdot 71+0\cdot 1\cdot 71}{71}=0$
\end{longtable}

\noindent and

\begin{longtable}{l}
$\displaystyle b=\frac{1-(1+9\cdot 0)}{71}=0.$
\end{longtable}

\noindent Finally, from Lemma 2.1,

\begin{longtable}{l}
$\displaystyle \gamma = \frac{71\cdot 0+9\cdot 1+0\cdot 10+(71\cdot 0+1+9\cdot 0)\sqrt {10}}{1}=9+\sqrt{10}.$
\end{longtable}

\noindent Since $N_{K/\mathbb Q}(\gamma)=71$,  we have $(71, 9+\sqrt{10})=(\gamma)$ by (3.2).

\end{example}
 
\begin{example}

\normalfont Consider $K={\mathbb Q}(\sqrt {-5})$. Since $-5\not \equiv 0\pmod{47}$ and $(18)^2+5=7\cdot47$, the prime 47 splits in $K$. We have $47{\mathcal O}_K=(47, 18+\sqrt {-5})(47, 18-\sqrt {-5})$. We will now determine if $(47, 18+\sqrt {-5})$ is principal, and, if so, find a generator for it. Since $-5\equiv 3\pmod 4$ we have $\delta=1$, so following Theorem 1.1 (i) we consider the equation $7x^2+2(18)xy+47y^2=1$. Suppose that $x=a$ and $y=b$ is a solution in integers to this equation. Then $a$ and $b$ must be relatively prime. Moreover,  $f(x,y)=7x^2+2(18)xy+47y^2$ must be properly equivalent to $g(x,y)=x^2+5y^2$ by art. 155 and art. 168 of [5]. However, $f(1,0)=7$ whereas $g(x,y)=7$ has no integer solutions since the congruence $x^2+y^2\equiv 3\pmod 4$ has none. It follows that $f$ cannot be properly equivalent to $g$ by art. 166 of [5]. Therefore $7x^2+2(18)xy+47y^2=1$ has no solutions in integers $x$ and $y$. Hence, by Theorem 1.1 (i) the ideal $(47, 18+\sqrt{-5})$ is not principal. 

\end{example}

\begin{example}

\normalfont Consider $K={\mathbb Q}(\sqrt {-23})$. Since $-23\not \equiv 0\pmod 3$ and $1^2+23=8\cdot3$, the prime 3 splits in $K$. We have $3{\mathcal O}_K=(3, 1+\sqrt {-23})(3, 1-\sqrt {-23})$. Since $-23\equiv 1\pmod 4$ we have $\delta=2$. Hence, by Theorem 1.1 (i), $(3, 1+\sqrt{-23})$ is principal if and only if $f(x,y)=8x^2+2(1)xy+3y^2$ represents 4. First, notice that there do not exist integers $a$ and $b$, necessarily relatively prime, such that $f(a,b)=1$. For if so, then by art. 155 and art. 168 of [5], $f$ is properly equivalent to $g(x,y)=x^2+23y^2$. However, this is not possible by art. 166 of [5] since $f(0,1)=3$, but $g$ cannot represent 3. Now suppose there are integers $a$ and $b$, not relatively prime, such that $f(a,b)=4$. Then the greatest common divisor of $a$ and $b$ must be 2. Hence, $a=2a'$ and $b=2b'$ where $a'$ and $b'$ are relatively prime. But then $f(a',b')=1$, a contradiction. So if there are integers $a$ and $b$ such that $f(a,b)=4$, then $a$ and $b$ are relatively prime. Suppose this is the case. Then by art. 155 and art. 168 of [5] either $f$ is properly equivalent to $h(x,y)=4x^2+2(1)xy+6y^2$, or $f$ is properly equivalent to $k(x,y)=4x^2+2(-1)xy+6y^2$. But neither is possible by art. 161 of [5] since the greatest common divisor of the coefficients of $h$ is 2, and likewise for $k$, but 2 does not divide each coefficient of $f$. So $f$ cannot represent 4. Hence, $(3, 1+\sqrt{-23})$ is not principal. This can also be shown as follows. We have $f(x,y)=8x^2+2(1)xy+3y^2$ represents 4 if and only if $F(x,y)=3x^2+2(1)xy+8y^2$ represents 4. Suppose there are integers $a$ and $b$ such that $F(a,b)=4$. Then by Lemma 6.1 below, $b=0$. But then $3a^2=4$, a contradiction. Hence $f$ cannot represent 4, so $(3, 1+\sqrt{-23})$ is not principal.

\end{example}

\section{Imaginary Quadratic Fields with Class Number 1}

\noindent  Using the methods illustrated in the examples above, together with Theorem 6.10 below and some additional results, we can give a proof, up to the possible existence of at most one additional field, of Theorem 1.3.  

\begin{lemma}

Let $f(x,y)=ax^2+2bxy+cy^2$ be a binary quadratic form with integer coefficients such that $a>0$ and $d=b^2-ac<0$. Suppose $M,r$, and $s$ are integers such that $M>0$ and $f(r,s)\leq M$. Then

\begin{longtable}{l}
$\displaystyle -\sqrt{\frac{Ma}{|d|}}\leq s \leq \sqrt{\frac{Ma}{|d|}}.$
\end{longtable}

\end{lemma}

\begin{proof}

Following the proof of Theorem 6.24, p. 340 of [15] we have

\begin{longtable}{l}
$\displaystyle f(r,s)=ar^2+2brs+cs^2=\frac{1}{4a}[(2ar+2bs)^2+4|d|s^2].$
\end{longtable}

\noindent So if $f(r,s)\leq M$, then

\begin{longtable}{l}
$(2ar+2bs)^2+4|d|s^2\leq 4Ma$
\end{longtable} 

\noindent which implies

\begin{longtable}{l}
$4|d|s^2\leq 4Ma$
\end{longtable}

\noindent if and only if

\begin{longtable}{l}
$\displaystyle s^2\leq \frac{Ma}{|d|}.$
\end{longtable}

\noindent Hence,

\begin{longtable}{l}
$\displaystyle -\sqrt{\frac{Ma}{|d|}}\leq s \leq \sqrt{\frac{Ma}{|d|}}.$
\end{longtable}

\end{proof}

\begin{proposition}
 Let $D<0$ be a square-free integer such that $|D|>2$, and either $D\equiv 2 \pmod 4$ or $D\equiv 3 \pmod 4$. Let $K={\mathbb Q}(\sqrt D)$. Then $K$ has class number greater than 1.
 
 \end{proposition}
 
\begin{proof}
 
First assume $D\equiv 2 \pmod 4$. If $4+|D|$ is not divisible by an odd prime then $4+|D|=2^m$ for some positive integer $m$. But then we have $2\equiv 2^m\pmod 4$ so $m=1$, a contradiction. Hence, there are positive integers $l$ and $q$ such that $q$ is an odd prime and $4+|D|=lq$. Since $D\not \equiv 0\pmod q$ and $2^2-D=lq$, we have $q{\mathcal O}_K=(q, 2+\sqrt D)(q, 2-\sqrt D)$. Since $2\equiv lq \pmod 4$ it follows that  $l \not \equiv 1 \pmod 4$. Also, we must have $l<|D|$ for if $l\geq |D|$ then $4+|D|=lq\geq|D|q\geq |D|3$ so $4\geq |D|2$ and hence $2\geq|D|$ which is a contradiction. Thus $1<l<|D|$. Now consider $f(x,y)=lx^2+2(2)xy+qy^2$. By Theorem 1.1 (i) the ideal $(q, 2+\sqrt D)$ is principal if and only if $f$ represents 1. Suppose there are integers $a$ and $b$ such that $f(a,b)=1$. Then by Lemma 6.1, $b=0$. So $la^2=1$ which is a contradiction since $l>1$. Hence, $(q, 2+\sqrt D)$ is not principal, so $K$ has class number greater than 1. An analogous argument proves the other case. Assume $D\equiv 3 \pmod 4$. Then $|D|\equiv 1 \pmod 4$. If $1+|D|$ is not divisible by an odd prime then $1+|D|=2^m$ for some positive integer $m$. But then we have $2\equiv 2^m\pmod 4$ so $m=1$, a contradiction. Hence, there are positive integers $l$ and $q$ such that $q$ is an odd prime and $1+|D|=lq$. Since $D\not \equiv 0\pmod q$ and $1^2-D=lq$ we have $q{\mathcal O}_K=(q, 1+\sqrt D)(q, 1-\sqrt D)$.
Since $2\equiv lq \pmod 4$ it follows that $l \not \equiv 1 \pmod 4$. Also, we must have $l<|D|$ for if $l\geq |D|$ then $1+|D|=lq\geq|D|q\geq |D|3$ so $1\geq |D|2$ and hence $1/2\geq|D|$ which is a contradiction. So we have $1<l<|D|$. Now consider $f(x,y)=lx^2+2(1)xy+qy^2$. By Theorem 1.1 (i) the ideal $(q, 1+\sqrt D)$ is principal if and only if $f$ represents 1. Suppose there are integers $a$ and $b$ such that $f(a,b)=1$. Then by Lemma 6.1, $b=0$. So $la^2=1$ which is a contradiction since $l>1$. Hence, $(q, 1+\sqrt D)$ is not principal, so $K$ has class number greater than 1.
 
\end{proof}

\begin{corollary}

Let $D<0$ be a square-free integer such that either $D\equiv 2 \pmod 4$ or $D\equiv 3 \pmod 4$. Then $K={\mathbb Q}(\sqrt D)$ has class number 1 if and only if $D=-1$ or $D=-2$.

\end{corollary}

\begin{proof}

By Proposition 6.2 if $K$ has class number 1, then $D=-1$ or $D=-2$. Conversely, if $D=-1$ or $D=-2$, then $K$ is norm-Euclidean so $K$ has class number 1 (see, for instance, Theorem 10.38, p. 678 and Corollary 10.38.1, p. 680 of [15]).

\end{proof}

It remains to consider the fields $K={\mathbb Q}(\sqrt D)$  such that $D<0$ is square free and $D\equiv 1\pmod 4$. 

\begin{proposition} 

Let $D<0$ be a square-free integer such that $D\equiv 1\pmod 4$ and  $3\leq |D|\leq 15$. Then $K={\mathbb Q}(\sqrt D)$ has class number 1 if and only if $D\in \{-3,-7,-11\}$. 
 
 \end{proposition}
 
 \begin{proof}
 
We have $D\in\{-3,-7,-11,-15\}$. The fields $K$ with $D\in \{-3,-7,-11\}$ are norm-Euclidean so have class number 1 (see, for instance, Theorem 10.38, p. 678 and Corollary 10.38.1, p. 680 of [15]). Since 15 is the product of two distinct primes the field ${\mathbb Q}(\sqrt{-15})$ has class number greater than 1 by genus theory (see, for instance, Corollary 1, p. 446 of [11]).
 
 \end{proof}

\begin{proposition}

Let $D<0$ be a square-free integer  such that  $D\equiv 1 \pmod 4$. Let $K={\mathbb Q}(\sqrt D)$ and assume $|D|>16$. Then $K$ has class number 1 only if  $|D|=4p-1$ where $p$ is a prime such that $p\geq 5$, and $4p-1$ and $4p+3$ are prime. Moreover, if $p>5$, then either $p\equiv 1 \pmod{10}$ or  $p\equiv 7 \pmod {10}$, so $4p-1\equiv 3\pmod{10}$ or $4p-1\equiv 7\pmod{10}$.
 
\end{proposition}

\begin{remark}

\normalfont Let $D<0$ be a square-free integer such that $D\equiv 1 \pmod 4$. Let $K={\mathbb Q}(\sqrt D)$. By genus theory we know that if $|D|$ is not a prime, then $K$ has class number greater than 1 (see, for instance, Corollary 1, p. 446 of [11]). However, we do not make use of this fact until Case 2.2.1 of the proof of Proposition 6.5.

\end{remark}
 
\begin{proof} 

Since $4+|D|\equiv 3\pmod 4$ we have $4+|D|=lq$ where $l$ and $q$ are odd positive integers  and $q$ is prime. Since $D\not \equiv 0\pmod q$ and $2^2-D=lq$ we have $q{\mathcal O}_K=(q, 2+\sqrt D)(q, 2-\sqrt D)$.

\begin{case}
\normalfont Assume $q\equiv 1\pmod 4$. Then $l\equiv 3\pmod 4$, so $l$ is not a square. Suppose that $4l\geq |D|$. Then $4+|D|=lq\geq \frac{1}{4} |D|q\geq \frac{1}{4} |D|5$. Hence, $4\geq \frac{1}{4} |D|$ so $16\geq |D|$, which is a contradiction. Therefore, $4l<|D|$. Now let $f(x,y)=lx^2+2(2)xy+qy^2$ and suppose there are integers $a$ and $b$ such that $f(a,b)=4$. Then by Lemma 6.1, $b=0$. But then $la^2=4$ which contradicts the fact that $l$ is not a square. Hence $f$ cannot represent 4. So by Theorem 1.1 (i) the ideal $(q, 2+\sqrt D)$ is not principal, so $K$ has class number greater than 1. 
\end{case}

\begin{case}
\normalfont Assume $q\equiv 3\pmod 4$. 
\end{case}

\begin{subcase}
\normalfont Suppose $q<|D|$. Then $4+|D|=lq<l|D|$ so $4<(l-1)|D|$ which implies $l>1$. Moreover, since $l\equiv 1\pmod 4$ we have $l\geq 5$. We will now show that $f(x,y)=lx^2+2(2)xy+qy^2$ cannot represent 4. Let $F(x,y)=qx^2+2(2)xy+ly^2$.  Note that $f$ represents 4 if and only if $F$ represents 4. We will show that $F$ cannot represent 4. Suppose $4q\geq |D|$. Then $4+|D|=lq\geq l\cdot\frac{1}{4}|D|\geq \frac{5}{4}|D|$. Hence, $4\geq \frac{1}{4} |D|$ so $16\geq |D|$ which is a contradiction. Therefore, $4q<|D|$. Now suppose there are integers $a$ and $b$ such that $F(a,b)=4$. Then by Lemma 6.1, $b=0$. But then $qa^2=4$ which is a contradiction. Hence, $F$ cannot represent 4 and so $f$ cannot represent 4. So by Theorem 1.1 (i) the ideal $(q, 2+\sqrt D)$ is not principal, so $K$ has class number greater than 1.
\end{subcase}

\begin{subcase}
\normalfont Suppose $q\geq |D|$. Since $D\not \equiv 0\pmod q$ we have $q> |D|$. Then $4+|D|=lq>l|D|$ so $4>(l-1)|D|$. Since $|D|>16$ this implies $l=1$. So in this case $f(2,0)=4$ where $f(x,y)=x^2+2(2)xy+qy^2$. By Theorem 1.1 (i) the prime ideal $(q, 2+\sqrt D)$ is principal. So in this case we need to make a modification in the hopes of obtaining an ideal of ${\mathcal O}_K$ that is not principal. To begin with, note that $1+|D|=q-3$. There are two cases to consider.

\begin{subsubcase}\label{subsubcase}
 \normalfont Assume $q-3=2^m$. If $|D|$ is not a prime, then by genus theory $K$ has class number greater than 1 (see, for instance, Corollary 1, p. 446 of [11]). So assume $|D|$ is a prime. Since $|D|>16$ and $|D|=2^m-1$ we have $m\geq5$. Note that $9+|D|=2^m+8=2^3(2^{m-3}+1)$, $m\geq 5$. We now show that $2^{m-3}+1<|D|$. Suppose not. Then $2^{m-3}+1\geq 2^m-1$ so $2\geq 2^m-2^{m-3}$ which implies $2\geq 2^{m-3}(2^3-1)$, $m\geq5$, a contradiction. So $2^{m-3}+1<|D|$. Now let $p$ be an odd prime dividing $2^{m-3}+1$. Then $9+|D|=dp$ where $d\equiv 0\pmod 8$ and $p<|D|$. Note that if $D\equiv 0\pmod p$, then $p=3$, so $|D|=3$ since $|D|$ is a prime. But this contradicts the fact that $|D|>16$. Hence,  $D\not \equiv 0\pmod p$. Since $3^2-D=dp$ we have $p{\mathcal O}_K=(p, 3+\sqrt D)(p, 3-\sqrt D)$. Suppose that $4p\geq |D|$. Then $9+|D|=dp\geq d\cdot\frac{1}{4}|D|\geq 8\cdot \frac{1}{4}|D|=2|D|$ which implies $9\geq |D|$, a contradiction. So $4p<|D|$. Now consider $f(x,y)=dx^2+2(3)xy+py^2$ and $F(x,y)=px^2+2(3)xy+dy^2$.
If there are integers $a$ and $b$ such that $F(a,b)=4$, then by Lemma 6.1, $b=0$. Then $pa^2=4$, which is a contradiction. Hence, $F$ cannot represent 4. Since $f$ represents 4 if and only if $F$ represents 4, we have that $f$ does not represent 4. So by Theorem 1.1 (i) the ideal $(p, 3+\sqrt D)$ is not principal, so $K$ has class number greater than 1.

\begin{subsubcase}
\normalfont Assume $q-3\not=2^m$. Then $1+|D|=dp$ where $d$ and $p$ are positive integers with $p$ an odd prime and $d\equiv 0\pmod 4$.  Note that $D\not \equiv 0\pmod p$. Since $1^2-D=dp$ we have $p{\mathcal O}_K=(p, 1+\sqrt D)(p, 1-\sqrt D)$. Assume $d\ne 4$. Then $d\geq 8$. Hence, if $4p\geq |D|$, then $1+|D| =dp\geq d\cdot \frac{1}{4}|D|\geq 8\cdot \frac{1}{4}|D|=2|D|$ which implies $1\geq |D|$, a contradiction. So $4p<|D|$. Now consider $f(x,y)=dx^2+2(1)xy+py^2$ and $F(x,y)=px^2+2(1)xy+dy^2$.
If there are integers $a$ and $b$ such that $F(a,b)=4$, then by Lemma 6.1, $b=0$. Then $pa^2=4$, which is a contradiction. Hence, $F$ cannot represent 4. Since $f$ represents 4 if and only if $F$ represents 4, we have that $f$ does not represent 4. So by Theorem 1.1 (i) the ideal $(p, 1+\sqrt D)$ is not principal, so $K$ has class number greater than 1.
\end{subsubcase}
\end{subsubcase}
\end{subcase}

Thus, ${\mathbb Q}(\sqrt D)$ has class number 1 only if  $d=4$, so $1+|D|=q-3=4p$. Hence, $|D|=4p-1$, where $p$ is an odd prime,  $4p-1$ is a prime, and $4p+3$ is prime. Since $|D|>16$ we have $p\geq 5$. Moreover, if $p>5$, then $p\equiv 1, 3, 7, \text {or}\  9 \pmod {10}$. If $p\equiv 3 \pmod {10}$, then $4p+3\equiv 5 \pmod {10}$, a contradiction since  $4p+3$ is a prime greater than 5. If $p\equiv 9 \pmod {10}$, then $4p-1\equiv 5 \pmod {10}$, a contradiction since  $4p-1$ is a prime greater than 5. Hence $p\equiv 1 \ \text {or}\  7 \pmod {10}$, so  $4p-1\equiv 3 \ \text {or}\  7 \pmod {10}$.

\end{proof}

\begin{remark}

\normalfont With the notation as in the last paragraph of the proof of Proposition 6.5, note that $D\not \equiv 0\pmod p$. Since $1^2-D=4p$ we have  $p{\mathcal O}_K=(p, 1+\sqrt D)(p, 1-\sqrt D)$. Since $f(1,0)=4$ where $f(x,y)=4x^2+2(1)xy+py^2$, the ideal $(p,1+\sqrt D)$ is principal by Theorem 1.1 (i).  

\end{remark}

\begin{corollary}

Let $D<0$ be a square-free integer such that $D\equiv 1\pmod 4$ and  $16< |D|\leq 163$. Then $K={\mathbb Q}(\sqrt D)$ has class number 1 if and only if $D\in \{-19,-43,-67,-163\}$. 

\end{corollary}

\begin{proof}

Suppose $K$ has class number 1. Then by Proposition 6.5 we have $|D|=4p-1$ where $p$ is an odd prime such that $p\geq 5$, and $4p-1$ and $4p+3$ are prime. The only values of $|D|$ with $16< |D|\leq 163$ satisfying these conditions are the values $19,43,67$, and $163$. Conversely, each of the fields $K$ with $D\in \{-19,-43,-67,-163\}$ has class number 1 by Theorem 6.10 below.

\end{proof}

The next prime following $163$ of the form $4p-1$ where $p$ and $4p+3$ are prime is $907=4\cdot 227-1$. We will now prove the following theorem.

\begin{theorem} 

Suppose $p, 4p-1$, and $4p+3$ are prime and $4p-1\geq 907$. Then $K=\mathbb{Q}(\sqrt{1-4p})$ has class number greater than 1, with at most one possible exception.

\end{theorem}

The following result, proved in [4] and [14] and stated as Theorem 8.28 on page 452 of [11], will be used in our proof of Theorem 6.9. For a number field $K$ let $h(K)$ denote its class number.

\begin{theorem}[Frobenius--Rabinowitsch] 

Let $K$ be an imaginary quadratic field with discriminant $d\ne -3,-4,-8$. Then $h(K)=1$ holds if and only if $d\equiv 1\pmod 4$, and for $x=1,2,\ldots, (1-d)/4-1$ the polynomial 

\begin{longtable}{l}
$\displaystyle F_d(X)=X^2-X+\frac{1-d}{4}$
\end{longtable}

\noindent attains exclusively prime values.

\end{theorem}
Assuming a certain condition, we can use an argument similar to that used to prove case 1 of Theorem 1, p. 1867 of [6], to produce for each prime $p$ in the statement of Theorem 6.9 a prime $q<p$ such that $q\equiv 1\pmod{4}$ and $4p-1$ is a quadratic residue modulo $q$. This will then allow us to prove Theorem 6.9 using Theorem 6.10 and a theorem of Tatuzawa. We will need the following three well-known results. It will be convenient to include their proofs here.

\begin{lemma}

Suppose $x$ and $y$ are odd integers such that $|x|\neq |y|$. Then $x^2+y^2$ is divisible by a prime $q$ such that $q\equiv 1\pmod{4}$.

\end{lemma}

\begin{proof}

Let $d$ be the greatest common divisor of $x$ and $y$. Then $x^2+y^2=d^2(a^2+b^2)$ for relatively prime odd integers $a$ and $b$ and $|a|\neq |b|$. We may assume that $|a|\geq 1$ and $|b|\geq 3$, so $a^2+b^2\geq 10$. Since $a$ and $b$ are odd we have $a^2+b^2\equiv 2\pmod{4}$. Thus, $2\mid a^2+b^2$, but $4\nmid a^2+b^2$. Hence, $a^2+b^2= 2n$ where $n$ is an odd integer and $n\geq 5$. Let $q$ be a prime divisor of $n$. Then $q$ is odd and $a^2+b^2\equiv 0\pmod{q}$, so $a^2\equiv -b^2\pmod{q}$. If $q\mid b$ then $q\mid a$ which is a contradiction. Hence, $q\nmid b$, so $(ab^{-1})^2\equiv -1\pmod{q}$. Thus, $-1$ is a quadratic residue modulo $q$ and so $q\equiv 1\pmod{4}$ (see, for instance, Corollary 3.2.1, p. 154 of [15]).

\end{proof}

\begin{lemma}

Suppose $\ell$ is a prime such that $\ell\equiv 3\pmod{8}$. Then $2x^2+y^2=\ell$ has exactly four solutions $(x,y)\in \mathbb{Z}^2$ and for each of these solutions $(x,y)$, both $x$ and $y$ are odd.

\end{lemma}

\begin {proof}

Since $\ell\equiv 3\pmod{8}$, we have $-2\equiv n^2\pmod{\ell}$ for some $n\in \mathbb{Z}$ (see, for instance, Corollary 3.2.1, p. 154 and Theorem 3.4, p. 156 of [15]) so the principal ideal $(\ell)$ in the ring of integers $\mathbb{Z}[\sqrt{-2}]$ of $\mathbb{Q}(\sqrt{-2})$ splits uniquely into the product of two prime ideals in $\mathbb{Z}[\sqrt{-2}]$,
\[(\ell)=(\ell, n+\sqrt{-2})(\ell, n-\sqrt{-2})\]
(see, for instance, Theorem 25, p. 74 of [10]). Since $\mathbb{Z}[\sqrt{-2}]$ is a principal ideal domain by Corollary 6.3 we have $(\ell, n+\sqrt{-2})=(a+b\sqrt{-2})$ for some $a,b\in \mathbb{Z}$, so $(\ell, n-\sqrt{-2})=(a-b\sqrt{-2})$. Hence, $(\ell)=(a^2+2b^2)$. Since the group of units in $\mathbb{Z}[\sqrt{-2}]$ is $\{\pm 1\}$ we have $\ell=a^2+2b^2$. Therefore $\ell=2x^2+y^2$ if and only if $(x,y)$ equals one of the four distinct ordered pairs $(b,a),(-b,a),(b,-a)$, or $(-b,-a)$. Since  $\ell\equiv 3\pmod{4}$ and the square of an integer is congruent to $0$ or $1$ modulo $4$, both $a$ and $b$ must be odd.

\end{proof}

\begin{lemma}

Suppose $\ell$ is a prime such that $\ell\equiv 3\pmod{8}$ and $\ell>3$. Then $x^2+y^2+z^2=\ell$ has exactly $24+48k$ solutions $(x,y,z)\in \mathbb{Z}^3$ where $k$ is the number of solutions $(x,y,z)$ such that $0<x<y<z$. Moreover, for any solution $(x,y,z)$ we have $x$, $y$, and $z$ are odd.

\end{lemma}

\begin{proof}

Suppose $(a,b,c)\in \mathbb{Z}^3$ is a solution to $x^2+y^2+z^2=\ell$. Then, since $\ell \equiv 3\pmod{4}$ and the square of an integer is congruent to $0$ or $1$ modulo $4$, it follows that $a,b$, and $c$ are odd. Since $\ell$ is a prime and $\ell>3$ there can be no solutions $(a,b,c)\in\mathbb{Z}^3$ such that $|a|=|b|=|c|$. By Lemma 6.12 there is a solution $(a,b,c)\in\mathbb{Z}^3$ such that exactly two of the values $|a|,|b|$, and $|c|$ are equal. We may assume that $|a|=|b|$. By taking all possible permutations of the coordinates of the solution $(a,a,c)$ we obtain the three solutions $(a,a,c), (a,c,a)$, and $(c,a,a)$.
 By taking all possible sign combinations of the coordinates of these three solutions we obtain $3\cdot 8=24$ solutions. These $24$ solutions must account for all the solutions $(x,y,z)$ for which exactly two of the values $|x|,|y|$, and $|z|$ are equal. For suppose $(u,v,w)$ is an additional such solution, say, $|u|=|v|$. Then $2u^2+w^2=\ell$ which implies by Lemma 6.12 that $(u,w)\in\{(a,c),(-a,c),(a,-c),(-a,-c)\}$, a contradiction. All other solutions $(x,y,z)$, if any, are such that $|x|,|y|$, and $|z|$ are distinct. From such a solution we can always obtain, by permutation or sign changes of coordinates, a solution $(x,y,z)$ such that $0<x<y<z$. Let $k$ be the number of such solutions. By taking all permutations of the coordinates of each of these $k$ solutions we obtain a total of $6k$ solutions. Then, by taking all possible sign combinations of the coordinates of these $6k$ solutions we obtain $6k\cdot 8=48k$ solutions.

\end{proof}
\begin{proof}[Proof (of Theorem 6.9)]

Suppose $p, 4p-1$, and $4p+3$ are prime and $4p-1\geq 907$. Since $p\equiv 1,3,5, \text{or}\ 7\pmod{8}$ we have $4p-1\equiv 3\pmod{8}$. By Lemma 6.13 the equation $x^2+y^2+z^2=4p-1$ has exactly $24+48k$ solutions $(x,y,z)\in \mathbb{Z}^3$ where $x,y$, and $z$ are odd and $k$ is the number of solutions such that $0<x<y<z$. Suppose $k\geq 1$. Then $x^2+y^2+z^2=4p-1$ for some $(x,y,z)\in \mathbb{Z}^3$ where $|x|,|y|$, and $|z|$ are odd and distinct. Moreover, since $p$ and $4p-1$ are prime, and $4p-1\geq 907$, we must have $p\equiv 2\pmod{3}$. Hence, $4p-1\equiv 1\pmod{3}$. Since the square of an integer is congruent to $0$ or $1$ modulo $3$
it follows that exactly two of $x,y$, and $z$ are congruent to $0$ modulo 3. We can assume that $x=3a$ and $y=3b$ where $|a|$ and $|b|$ are distinct odd integers. Since $a$ and $b$ are odd we have $a^2+b^2\equiv 2\pmod{4}$ so $2\mid a^2+b^2$, but $4\nmid a^2+b^2$. By Lemma 6.11 there is a prime $q$ such that $q\equiv 1\pmod{4}$ and $q\mid x^2+y^2$. Hence, $4p-1$ is a quadratic residue modulo $q$. Moreover, since
\[x^2+y^2=9(a^2+b^2)=18\frac{a^2+b^2}{2}\]
it follows that $q\mid \frac{a^2+b^2}{2}$. Therefore,
\[4p-1=x^2+y^2+z^2=18\frac{a^2+b^2}{2}+z^2\geq18q+z^2>18q\]
so $q<\frac{4p-1}{18}<p$. Now consider the congruence
\[X^2-X+p\equiv 0\pmod{q}.\]
Computing the Legendre symbol and using the fact that $q\equiv 1\pmod{4}$ we obtain 
\[\left(\frac{1-4p}{q}\right)=\left(\frac{-1}{q}\right)\left(\frac{4p-1}{q}\right)=\left(\frac{4p-1}{q}\right)=1.\]
Since $p\not\equiv 0\pmod{q}$, the quadratic formula gives an $n\in \{1, 2, 3, \ldots, q-1\}$ that is a solution to the congruence. So $n^2-n+p\equiv 0\pmod q$. Since $n^2-n+p=n(n-1)+p\geq p>q$, it follows that $n^2-n+p$ is not prime. Furthermore, $1\leq n\leq q-1<p-1$. Hence, by Theorem 6.10, the field ${\mathbb Q}(\sqrt{1-4p})$ has class number greater than 1. We now show that for all primes $4p-1$ such that $p$ and $4p+3$ are prime and $4p-1\geq 907$,
we have $k\geq1$ with at most one possible exception. By Lemma 6.13 above and Theorem B, p. 99 of [3], and using Theorem 4, p.54 of [8] and (4.13),  p. 57 of [8], we have
\[24+48k=\frac{24\sqrt{4p-1}}{\pi}L(1,\chi_{1-4p}),\]
where $\chi_{1-4p}$ is the real primitive character defined in terms of the Kronecker symbol by $\chi_{1-4p}(m)=\left(\frac{1-4p}{m}\right)$. So
\[k=\frac{\sqrt{4p-1}}{2\pi}L(1,\chi_{1-4p})-\frac{1}{2}.\]
Therefore, $k\geq 1$ if and only if $k>0$ if and only if 
\[L(1,\chi_{1-4p})>\frac{\pi}{\sqrt{4p-1}}.\] 
By Theorem 2, p. 164 of [18], with $\epsilon =1/11.2$, so 
$\lceil e^{11.2}\rceil=73131$, we obtain for all $4p-1\geq 73131$, with at most one possible exception,
\begin{equation}
L(1,\chi_{1-4p})>\frac{0.655}{11.2(4p-1)^{1/11.2}}.
\end{equation}
Since 
\[\frac{0.655}{11.2(4p-1)^{1/11.2}}\div \frac{\pi}{\sqrt{4p-1}}=\frac{0.655(4p-1)^{23/56}}{(11.2)\pi}\]
is a strictly increasing function of $4p-1$ and 
\[\frac{0.655(73131)^{23/56}}{(11.2)\pi}=1.851950536\ldots,\] 
we have for all $4p-1\geq 73131$, with at most one possible exception,
\[L(1,\chi_{1-4p})>\frac{\pi}{\sqrt{4p-1}}.\] 
Finally, one easily verifies by use of a computer (using, for instance, Python (SymPy)) that there are $59$ primes of the form $4p-1$ where $p$ and $4p+3$ are prime such that $907\leq 4p-1< 73131$, and for all such primes we have $k\geq1$. These primes are listed in Table 6.1 below along with one solution $(x,y,z)$ to $x^2+y^2+z^2=4p-1$ satisfying $0<x<y<z$. 

\end{proof}

\begin{longtable}{lll}  
    \textbf{907}: (3, 13, 27) & \textbf{13963}: (3, 27, 115) & \textbf{42403}: (7, 105, 177) \\
    \textbf{1867}: (7, 27, 33) & \textbf{14323}: (3, 25, 117) & \textbf{44683}: (9, 139, 159) \\
    \textbf{2083}: (3, 7, 45) & \textbf{15643}: (15, 17, 123) & \textbf{46723}: (3, 115, 183) \\
    \textbf{2347}: (15, 21, 41) & \textbf{15667}: (3, 23, 123) & \textbf{47947}: (11, 99, 195) \\
    \textbf{2707}: (5, 9, 51) & \textbf{15787}: (1, 69, 105) & \textbf{50587}: (5, 51, 219) \\
    \textbf{3187}: (15, 19, 51) & \textbf{20323}: (1, 21, 141) & \textbf{51283}: (9, 159, 161) \\
    \textbf{3907}: (3, 33, 53) & \textbf{21187}: (3, 27, 143) & \textbf{52963}: (5, 87, 213) \\
    \textbf{5227}: (3, 27, 67) & \textbf{22027}: (3, 57, 137) & \textbf{53323}: (13, 123, 195) \\
    \textbf{5923}: (3, 17, 75) & \textbf{25243}: (21, 51, 149) & \textbf{57283}: (5, 33, 237) \\
    \textbf{6043}: (7, 45, 63) & \textbf{26947}: (3, 73, 147) & \textbf{57787}: (5, 99, 219) \\
    \textbf{6547}: (9, 15, 79) & \textbf{27427}: (3, 93, 137) & \textbf{58147}: (3, 113, 213) \\
    \textbf{7243}: (5, 57, 63) & \textbf{28603}: (3, 37, 165) & \textbf{58363}: (13, 45, 237) \\
    \textbf{7603}: (3, 5, 87) & \textbf{30427}: (9, 105, 139) & \textbf{58963}: (21, 151, 189) \\
    \textbf{7723}: (9, 51, 71) & \textbf{31267}: (1, 45, 171) & \textbf{59467}: (7, 57, 237) \\
    \textbf{8443}: (3, 53, 75) & \textbf{32323}: (11, 111, 141) & \textbf{61507}: (3, 73, 237) \\
    \textbf{9787}: (3, 47, 87) & \textbf{35323}: (3, 33, 185) & \textbf{61987}: (9, 135, 209) \\
    \textbf{10627}: (9, 39, 95) & \textbf{37507}: (3, 87, 173) & \textbf{62323}: (5, 57, 243) \\
    \textbf{11443}: (3, 45, 97) & \textbf{37963}: (15, 127, 147) & \textbf{62683}: (3, 157, 195) \\
    \textbf{11827}: (1, 45, 99) & \textbf{37987}: (3, 27, 193) & \textbf{72763}: (5, 117, 243) \\ 
    \textbf{13003}: (3, 15, 113) & \textbf{38707}: (13, 123, 153) &  \\ \\
     \caption*{Table 6.1}
  \end{longtable}

\begin{corollary} 

With the notation as in the statement of Theorem 6.9, suppose for $4p-1\geq 73131$ the $L$-function $L(s,\chi_{1-4p})$ has no zeros in the half-plane $\mathrm{Re}(s)>219/224$. Then the exceptional case of Theorem 6.9 does not exist.

\end{corollary}

\begin{proof}

Suppose the exception does exist. Then the bound $(6.1)$ fails for some $4p-1\geq 73131$, say, $4p_0-1$. Therefore, by Lemma 9, p. 173 of [18], $L(s,\chi_{1-4p_0})$ has a real zero $\beta\in(1-\epsilon/4,1)$ where $\epsilon=1/11.2$. Hence $219/224<\beta<1$, which contradicts the assumption that the half-plane $\mathrm{Re}(s)>219/224$  contains no such zeros. So $(6.1)$ holds for all  $4p-1\geq 73131$.

\end{proof}

\begin{remark}
\normalfont If the recent result claiming the half-plane $\mathrm{Re}(s)>11/12$ contains no zeros of any Dirichlet $L$-function is true (see, for instance, [12]), then, since $11/12<219/224$, the hypothesis of Corollary 6.14 is true.
\end{remark}

\section*{Acknowledgments}

\noindent The 59 primes and corresponding equation solutions listed in Table 6.1 were independently found by Gemini (an AI developed by Google) and 
Grok (an AI developed by xAI) both using Python and the SymPy library. The LaTeX code for Table 6.1 was generated by Gemini. The author verified that all 59 numbers listed are prime and satisfy the stated conditions, and verified the corresponding equation solutions. The author accepts full responsibility for any errors or omissions in the list.

\section*{References}

\begin{enumerate}

\item[{[1]}] Baker, A.: {\it Linear forms in the logarithms of algebraic numbers.} Mathematika 13 (1966), 204--216.

\item[{[2]}]  Baker, A.: {\it On the class number of imaginary quadratic fields.} Bull. Amer. Math. Soc. 77 (1971), 678--684.

\item[{[3]}] Bateman, P. T.: {\it On the representations of a number as the sum of three squares.} Trans. Amer. Math. Soc. 71 (1951), 70--101. 

\item[{[4]}] Frobenius, F. G.: {\it \"Uber quadratische Formen, die viele Primzahlen darstellen.} SBer. Kgl. Preu{\ss}. Akad. Wiss. Berlin (1912), 966--980.

\item[{[5]}]  Gauss, C. F.: Disquisitiones Arithmeticae. Springer-Verlag, New York (1986)

\item[{[6]}] Gica, A.: {\it Quadratic residues of certain types.} Rocky Mt. J. Math. 36 (2006), No. 6, 1867--1871.

\item[{[7]}]  Goldfeld, D.: {\it Gauss's class number problem for imaginary quadratic fields.} Bull. Amer. Math. Soc. (N.S.) 13 (1985), no. 1, 23--37.

\item[{[8]}] Grosswald, E.: Representations of Integers as Sums of Squares. Springer-Verlag, New York Inc. (1985)

\item[{[9]}]  Heegner, K.: {\it Diophantische Analysis und Modulfunktionen.} Math. Z. 56 (1952), 227--253.

\item[{[10]}]  Marcus, D. A.: Number Fields. Springer-Verlag, New York-Heidelberg (1977)

\item[{[11]}] Narkiewicz, W.: Elementary and Analytic Theory of Algebraic Numbers. Springer-Verlag, Berlin (2004)

\item[{[12]}] OpenAI: {\it The Quasi-Riemann Hypothesis.} OpenAI Math Release preprint, October 5, 2026, 1--49.

\item[{[13]}]  Pollack, P.: A Conversational Introduction to Algebraic Number Theory. Arithmetic Beyond $\mathbb Z$. American Mathematical Society, Providence, RI (2017)

\item[{[14]}]  Rabinowitsch, G.: {\it Eindeutigkeit der Zerlegung in Primzahlfaktoren in quadratischen Zahlk\"orpern.}  J. Reine Angew. Math. 142 (1913), 153--164. 

\item[{[15]}]  Redmond, D.: Number Theory: An Introduction. Marcel Dekker, Inc., New York (1996)

\item[{[16]}]  Stark, H. M.: {\it A complete determination of the complex quadratic fields of class-number one.} Michigan Math. J. 14 (1967), 1--27.

\item[{[17]}]  Stark, H. M.: {\it On the ``gap'' in a theorem of Heegner.} J. Number Theory 1 (1969), 16--27.

\item[{[18]}] Tatuzawa, T.: {\it On a theorem of Siegel.} Jap. J. Math. 21 (1951), 163--178. 

\end{enumerate}

\end{document}